\documentclass[11pt,twoside]{amsart}
\usepackage{amsmath, amsthm, amscd, amsfonts, amssymb, graphicx, color}
\usepackage{enumerate}
\newtheorem{thm}{Theorem}[section]
\newtheorem{cor}[thm]{Corollary}

\newtheorem{prop}[thm]{Proposition}

\numberwithin{equation}{section}
\def\pn{\par\noindent}

\newcommand{\eps}{\varepsilon}

\begin{document}

\title{Bounds for the modified eccentric connectivity index}
\author{Nilanjan De, Sk. Md. Abu Nayeem and Anita Pal}

\thanks{{\scriptsize
\hskip -0.4 true cm MSC(2010): Primary: 05C35; Secondary: 05C07, 05C40
}}
\maketitle



\begin{abstract}  The modified eccentric connectivity index of a graph is defined as the sum of the products of eccentricity with the total degree of neighboring vertices, over all vertices of the graph. This is a generalization of eccentric connectivity index. In this paper, we derive some upper and lower bounds for the modified eccentric connectivity index in terms of some graph parameters such as number of vertices, number of edges, radius, minimum degree, maximum degree, total eccentricity, the first and second Zagreb indices, Weiner index etc.
\end{abstract}

\vskip 0.2 true cm


\pagestyle{myheadings}
\markboth{\rightline {\sl N. De, S.M.A. Nayeem and A. Pal}}
         {\leftline{\sl Bounds for the modified eccentric connectivity index}}

\bigskip
\bigskip


\section{\bf Introduction}
\vskip 0.4 true cm

Let  $G=(V,E)$ be a simple connected graph with $n$ vertices and $m$ edges. We denote the degree of a vertex $v$ by $\deg(v)$ and the maximum and minimum degree of the graph $G$ by $\Delta$ and $\delta$ respectively. The distance between the vertices $u$ and $v$ of $G$, is equal to the length, that is the number of edges of a shortest path connecting $u$ and $v$ and we denote it by $d(u,v)$. Let $D(v)$ be the sum of distances from $v$ to all other vertices in $G$, i.e., $D(v)=\sum_{u\in V}d(v,u)$. For a given vertex $v$, its eccentricity $\eps(v)$  is the largest distance from $v$ to any other vertices of $G$. The radius and diameter of the graph are respectively the smallest and largest eccentricity among all the vertices of $G$, whereas, the total eccentricity of $G$, denoted by $\theta(G)$, is the sum of eccentricities of all the vertices of $G$. In recent years, different topological indices, based on degree and eccentricity of a vertex of a graph are subject to large number of studies. Among these, the eccentric connectivity index, proposed by Gupta et. al \cite{gup00}, defined as $\xi^c(G)=\sum_{v\in V} \deg(v)\eps(v)$, is one of the most popular vertex degree and eccentricity based topological index and is subject to a large number of chemical as well as mathematical studies \cite{ash11, de12b, zho10}. If $N(v)=\{v: uv=e\in E\}$, then the modified eccentric connectivity index of any graph is defined as

\begin{equation}
\xi_c(G)=\sum_{v\in V}\delta(v)\eps(v)
\label{eq1}
\end{equation}
where $\delta(v)=\sum_{u\in N(v)}\deg(u)$. It should be noted that $\delta(.)$ is used in the context of a vertex, whereas, $\delta$ is a global parameter of the graph $G$.

Ashrafi et. al in \cite{ash11} presented some graph operation of modified eccentric connectivity polynomial. However, no further study of this index is being recorded so far.
To establish upper and lower bounds of different topological indices in terms of different graph invariants and parameters, there are various study of which only some recent results are mentioned here \cite{ash11, de12a, de12b, de13a, de13b, ili11, mit70}. In this paper, first we find modified eccentric connectivity index of some particular graph, and then we investigate some new upper and lower bounds of modified eccentric connectivity index in terms of number of vertices $(n)$, number of edges $(m)$, maximum vertex degree $(\Delta)$, minimum vertex degree $(\delta)$, radius $(r)$, diameter $(d)$, total eccentricity $(\theta(G))$, the first Zagreb index $(M_1(G))$ \cite{ili11},  the second Zagreb index $(M_2(G))$ \cite{das04}, the first Zagreb eccentricity index $(E_1(G))$ \cite{xin11}, Wiener Index $(W(G))$ \cite{wag10, wu10}, Harary Index $(H(G))$ \cite{das09, zho08} and eccentric connectivity index $(\xi^c(G))$ \cite{zho10}.

\section{\bf Main Results}
\vskip 0.4 true cm

From (\ref{eq1}), it is clear that if all the vertices of $G$ are of same eccentricity $e$, then $\xi_c(G)=eM_1(G)$. Similarly, if all the vertices of $G$ are of same degree $(k)$ and eccentricity $(e)$ then $\xi_c(G)=nk^2e$. Using these results and from direct calculation, we can find the following explicit formulas for the eccentric connectivity index of different graphs.

\begin{prop}
Let $K_n,C_n,Q_m,\Pi_m,A_m$ denote the complete graph with $n$ vertices, the cycle on $n$ vertices, $m$-dimensional hypercube, $m$-sided prism and the $m$-sided antiprism respectively. Then the modified eccentric connectivity indices of these graphs are given as follows.
\begin{enumerate}[(i)]
\item $\xi_c(K_n)=n(n-1)^2$.
\item $\xi_c(C_n)=4n\lfloor n/2\rfloor$.
\item $\xi_c(Q_m)=m^32^m$.
\item $\xi_c(\Pi_m)=\left\{\begin{array}{ll}9m(m+2)&\mbox{when } m \mbox{ is even}\\9m(m+1)&\mbox{when } m \mbox{ is odd.}\end{array}\right.$
\item $\xi_c(A_m)=\left\{\begin{array}{ll}16m^2&\mbox{when } m \mbox{ is even}\\16m(m+1)&\mbox{when } m \mbox{ is odd.}\end{array}\right.$
\end{enumerate}
\end{prop}

The following results can also be computed from straight forward calculations.

\begin{prop}
Let $W_n$ and $B_n$ denote pyramid and bipyramid with  $n(\ge 3)$-gonal base, then
$\xi_c(W_n)=2n^2+ 5n$ and $\xi_c(B_n)=4n^2+32n$.
\end{prop}

\begin{prop}
Let $K_{m_1,m_2,\ldots,m_n}$ denotes the complete $n$-partite graph with $|V|=m_1+m_2+\ldots+m_n$ number of vertices $(m_i\ge 2, i=1,2,\ldots,n)$, then $$\xi_c(K_{m_1,m_2,\ldots, m_n})=2\sum_{i=1}^nm_i\left(\sum_{j=1,j\neq i}^nm_j(|V|-m_j)\right).$$
\end{prop}

\begin{prop}Let $S_n=K_{1,n-1}$ be a star graph with $n (\ge 3)$ vertices, then $\xi_c(S_n)=2n^2-3n+1$.
\end{prop}
We recall two fundamental indices, namely, the first and second Zagreb indices ($M_1(G)$ and $M_2(G)$) introduced by Gutman and Trinajsti\'{c} \cite{gut72}, are two of the oldest and most studied vertex degree based topological indices which are defined respectively as the sum of squares of the degrees of the vertices, and sum of product of the degrees of the adjacent vertices of a graph. Now first we calculate some upper and lower bounds of modified eccentric connectivity index in terms of some graph parameters such as maximum vertex degree ($\Delta$), minimum vertex degree ($\delta$), radius ($r$), diameter ($d$) etc.

\begin{thm}
Let $G=(V,E)$ be a simple connected graph with radius $r$ and diameter $d$, then
\begin{enumerate}[(i)]
\item $r\le \frac{\xi_c(G)}{M_1(G)}\le d$ and both hold with equality if and only if all the vertices of $G$ are of same eccentricity.
\item $r\delta^2\le \xi_c(G)\le d\Delta^2$ and both hold with equality if and only if $G$ is regular and all the vertices of $G$ are of same eccentricity.
\item $\delta^2\le\frac{\xi_c(G)}{\theta(G)}\le \Delta^2$ and both hold with equality if and only if $G$ is regular.
\end{enumerate}
\end{thm}

\begin{proof}
\begin{enumerate}[(i)]
\item Since we have, for any $v\in V(G), r\le \eps(v)\le d$  and $\displaystyle\sum_{v\in V}\delta(G)=M_1(G)$, from (\ref{eq1}), we get the desired result. Obviously, in this relation equality holds if and only if $r= \eps(v)= d$  for all $v\in V$.

\item Since, for any $v\in V$, $\delta\le \deg(v)\le \Delta$ and $r\le \eps(v)\le d$, so from (\ref{eq1}) the desired result follows. Clearly, the equality holds if and only if all the vertices are of same degree and eccentricity.

    \item Since, for any $v\in V$, $\delta\le \deg(v)\le \Delta$, so we have $\delta^2\le \delta(v)\le \Delta^2$. Thus from the definition of modified eccentric connectivity index we have, $\Delta^2\theta(G)\le \xi_c(G)\le \delta^2\theta(G)$, with equality when $G$ is a regular graph.
\end{enumerate}
\end{proof}

\subsection{Upper bounds}
In the following, we present some upper bounds for modified eccentric connectivity index of connected graphs.

\begin{thm}
Let $G$ be a simple connected graph, then $$\xi_c(G)\le \{2m-\delta(n-1)\}\theta(G)+(\delta-1)\xi^c(G)$$
and equality holds if and only if $G$ is a regular graph.
\end{thm}

\begin{proof}
Since we have, for any $v\in V(G), \delta(v)\le 2m-\deg(v)-(n-1-\deg(v))\delta$, $\xi_c(G)=\displaystyle \sum_{v\in V(G)}\delta(G)\eps(v)\le \sum_{v\in V(G)}\{2m-\deg(v)-(n-1-\deg(v))\delta\}\eps(v)\\=2m\sum_{v\in V(G)}\eps(v)-\sum_{v\in V(G)}\deg(v)\eps(v)-(n-1)\delta\sum_{v\in V(G)}\eps(v)+\delta\sum_{v\in V(G)}\deg(v)\eps(v)$, from where the desired result follows. Clearly in the above relation, equality holds if and only if $G$ is a regular graph.
\end{proof}

\begin{cor}
Let $G$ be a simple connected graph, then \[{\xi _c}(G) \le \left\{ {2m - \delta (n - 1)} \right\}({n^2} - 2m) + (2mn -{M_1}(G))(\delta  - 1)
\]
with equality if and only if $G \cong {K_n}$.
\end{cor}

\begin{proof} Since, for any $v \in V(G)$, $\eps (v) \le
n - \deg (v)$, with equality achieved for $G \cong {K_n} - je$ for $j =
1,2,\ldots,\lfloor n/2 \rfloor$ or $G \cong {P_n}$ \cite{hua12}, we have

${\xi ^c}(G) = \sum\limits_{v \in V(G)} {\deg (v)\eps (v)}  \le
\sum\limits_{v \in V(G)} {\deg (v)(n - \deg (v))}  = 2nm - {M_1}(G)$ and
$\theta (G) = \sum\limits_{v \in V(G)} {\eps (v)}  \le {n^2} - 2m$.

So, from the last theorem the desired result follows.
\end{proof}

Analogues to Zagreb indices, Ghorbani and Hosseinzadeh \cite{gho12} and Vuki\v{c}ev\'{i}c and Graovac \cite{vuk10} defined the Zagreb eccentricity indices by replacing degrees by eccentricity of the vertices, so that the first Zagreb eccentricity index $(E_{1}(G))$
is defined as sum of squares of the eccentricities of the vertices and the second
Zagreb eccentricity index $(E_{2}(G))$ is equal to sum of product of the
eccentricities of the adjacent vertices (see \cite{das13, de13a, vuk10}). Now we find bounds of
the modified eccentric connectivity index using these indices.

\begin{thm}
Let $G$ be a simple connected graph, then
\[
{\xi _c}(G) \le \sqrt {({\Delta ^2} + {\delta ^2}){M_1}(G){E_1}(G) - n{\Delta
^2}{\delta ^2}{E_1}(G)}
\]
with equality holds if and only if all the vertices are of
same degree and eccentricity.
\end{thm}

\begin{proof}
From Cauchy-Schwarz inequality, we have
\[
\sum\limits_{i = 1}^n {{x_i}{y_i}}  \le \sqrt {\sum\limits_{i = 1}^n {{x_i}^2}
\sum\limits_{i = 1}^n {{y_i}^2}.}
\]

Now putting ${x_i} = \delta ({v_i})$and ${y_i} = \eps ({v_i})$ for $i =
1,2,\ldots,n$, we have

\begin{equation}
{\xi _c}(G) = \sum\limits_{i = 1}^n {\delta ({v_i})\eps ({v_i})}  \le
\sqrt {\sum\limits_{i = 1}^n {\delta {{({v_i})}^2}} \sum\limits_{i = 1}^n
{\eps {{({v_i})}^2}} }  \le \sqrt {{E_1}(G)\sum\limits_{i = 1}^n {\delta
{{({v_i})}^2}}}
\label{eq:2}
\end{equation}
with equality if and only if all the vertices of $G$ are of same degree
and eccentricity.

Now, using the following Diaz-Metcalf inequality \cite{ash11}, we have if
$a_{i}$ and $b_{i}$, $i=1,2,\ldots,n$ are real
numbers such that $m{a_i} \le {b_i} \le M{a_i}$ for $i=1,2,\ldots,n$,
then
\begin{equation}
\sum\limits_{i = 1}^n {{b_i}^2}  + mM\sum\limits_{i = 1}^n {a_i^2 \le (m +
M)\sum\limits_{i = 1}^n {{a_i}{b_i}}. }
\label{eq:3}
\end{equation}

In the above relation, equality holds if and only if ${b_i} = m{a_i}$ or ${b_i} =
M{a_i}$ for every $i=1,2,\ldots,n$. By setting ${b_i} = \delta ({v_i})$
and ${a_i} = 1$, for $i=1,2,\ldots,n$, from above inequality we
get
\[
\sum\limits_{i = 1}^n {\delta {{({v_i})}^2}}  + mM\sum\limits_{i = 1}^n {{1^2}
\le (m + M)\sum\limits_{i = 1}^n {\delta ({v_i}).} }
\]

Since ${\delta ^2} \le \delta ({v_i}) \le {\Delta ^2}$ we have $m = {\delta
^2}$ and $M = {\Delta ^2}$, so that from above we get
\[
\sum\limits_{i = 1}^n {\delta {{({v_i})}^2}}  \le ({\Delta ^2} + {\delta
^2}){M_1}(G) - n{\Delta ^2}{\delta ^2}
\]
with the equality if and only if $\delta ({v_i}) = {\delta ^2} = {\Delta ^2}$ for
$i=1,2,\ldots,n$ i.e., $G$ is regular graph. Thus from (\ref{eq:2}) the
desired result follows. Obviously in this theorem equality holds if and only if
all the vertices are of same degree and eccentricity.
\end{proof}

\begin{thm}
Let $G$ be a simple connected graph, then
\[
{\xi _c}(G) \le n{M_1}(G) - 2{M_2}(G),
\]
with equality holds if and only if $G \cong {K_n} - je$ for $j
= 1,2,\ldots,\lfloor n/2\rfloor$ or $G \cong {P_n}$.
\end{thm}

\begin{proof}
We have for any $v \in V(G)$, $\eps (v) \le
n - \deg (v)$, with equality achieved for $G \cong {K_n} - je$ for $j
= 1,2,\ldots,\lfloor n/2\rfloor$ or $G \cong {P_n}$ \cite{hua12}. So from the definition of modified eccentric connectivity index
\[
\xi_c(G) = \sum\limits_{v\in V(G)}\delta(v)\eps(v) \le
\sum\limits_{v\in V(G)}\delta(v)(n - \deg(v)) = n\sum\limits_{v\in V(G)}
\delta(v) - \sum\limits_{v\in V(G)}\deg(v)\delta(v).
\]

Now, $\sum\limits_{v \in V(G)} {\deg (v)\delta (v) = 2{M_2}(G)}$ and the
desired result follows from above.
\end{proof}

\subsection{Lower bounds}
Now we find some lower bounds of modified eccentric connectivity index in terms
of maximum vertex degree ($\Delta$), minimum vertex degree ($\delta$), radius
(r), diameter ($d$), total eccentricity ($\theta(G)$), the first Zagreb
index $(M_{1}(G))$, the eccentric connectivity index (${\xi
^c}(G)$).

\begin{thm}
Let $G$ be a simple connected graph, then
\begin{enumerate}[(i)]
\item ${\xi _c}(G) \ge {M_1}(G)$ where, ${M_1}(G)$ is the first Zagreb index
of $G$ and equality holds if and only if $\eps (v) = 1$ for all
$v \in V(G)$ i.e., $G \cong K_n$.
\item ${\xi _c}(G) \ge {\xi^c}(G)$ where, ${\xi ^c}(G)$ is the
eccentric connectivity index of $G$ and equality holds if and only if $G\cong P_3$.
\end{enumerate}
\end{thm}

\begin{proof}
\begin{enumerate}[(i)]
\item Since $\eps(v) \ge 1$ for all $v \in V(G)$, from the
definition of modified eccentric connectivity index the desired result follows
with equality when $\eps(v) = 1$ i.e., $G\cong K_n$.

\item Again since $\delta(v) \ge \deg (v)$ for all $v \in V(G)$, we have from
definition of modified eccentric connectivity index ${\xi _c}(G) = \sum\limits_{v
\in V(G)}{\delta (v)\eps (v)}  \ge \sum\limits_{v \in V(G)} {\deg
(v)\eps (v)}  = {\xi ^c}(G)$. Clearly in this relation equality holds if
and only if $G$ is a path of length two.
\end{enumerate}
\end{proof}

\begin{thm}
Let $G$ be a simple connected graph, then
\[
{\xi _c}(G) \ge \frac{1}{{d{\Delta ^2} + r{\delta ^2}}}\left[ {{\Delta
^2}{\delta ^2}{E_1}(G) + \frac{{rd}}{n}{M_1}{{(G)}^2}} \right]
\]
and it holds with equality if and only if all the vertices of $G$ are of
same degree and eccentricity.
\end{thm}

\begin{proof}
Putting ${b_i} = \delta({v_i})$ and ${a_i} = \eps({v_i})$, for $i=1,2,\ldots,n$, in the Diaz-Metcalf inequality (\ref{eq:3}), we have
\[
\sum\limits_{i = 1}^n {\delta {{({v_i})}^2}}  + mM\sum\limits_{i = 1}^n
{\eps {{({v_i})}^2} \le (m + M)\sum\limits_{i = 1}^n {\delta
({v_i})\eps ({v_i})} }
\]

\begin{equation}
\mbox{i.e., } (m + M){\xi _c}(G) \ge \sum\limits_{i = 1}^n {\delta {{({v_i})}^2}}  +
mM{E_1}(G)
\label{eq:6}
\end{equation}

Now since $\frac{\delta^2}{d} \le \frac{\delta({v_i})}{\eps
({v_i})} \le \frac{\Delta^2}{r}$, from the equality condition of
Diaz-Metcalf inequality (\ref{eq:3}), we have $m = \frac{\delta^2}{d}$ and $M =
\frac{\Delta 2}{r}$.

Again from Cauchy-Schwartz inequality we have,
\begin{equation}
n\sum\limits_{i = 1}^n {\delta {{({v_i})}^2}}  \ge {\left[ {\sum\limits_{i =
1}^n {\delta ({v_i})} } \right]^2} = {M_1}{(G)^2}
\label{eq:7}
\end{equation}
with equality if and only if all the vertices of $G$ are of same degree,
so that from (\ref{eq:5}) we have
\[
(\frac{{{\delta ^2}}}{d} + \frac{{{\Delta ^2}}}{r}){\xi _c}(G) \ge
\frac{1}{n}{M_1}{(G)^2} + \frac{{{\Delta ^2}{\delta ^2}}}{{rd}}{E_1}(G)
\]
from where the desired result follows. Clearly equality achieved if and only if
all the vertices are of same degree and eccentricity.
\end{proof}

\begin{thm}
Let $G$ be a simple connected graph, then
\[
{\xi _c}(G) \ge \sqrt {\frac{1}{n}{M_1}(G){E_1}(G) - \frac{{{n^2}}}{4}(d{\Delta
^2} + r{\delta ^2}).}
\]
\end{thm}

\begin{proof}
Using the Ozeki's inequality \cite{oze68} we have, if
${a_1},{a_2},...,{a_m}$ and ${b_1},{b_2},...,{b_m}$ be positive real numbers
such that for $1 \le i \le n$, ${m_1} \le {a_i} \le {M_1}$ and ${m_2} \le
{b_i} \le {M_2}$ hold, then
\[
\left\{ {\sum\limits_{i = 1}^n {{a_i}^2} } \right\}\left\{ {\sum\limits_{i =
1}^n {{b_i}^2} } \right\} - {\left\{ {\sum\limits_{i = 1}^n {{a_i}{b_i}} }
\right\}^2} \le \frac{n^2}{4}{\left\{{M_1}{M_2} - {m_1}{m_2}\right\}^2}
\]

Now putting ${a_i} = \eps({v_i})$ and ${b_i} = \delta ({v_i})$ for
$i=1,2,\ldots,n$, so that ${m_1} = r$, ${M_1} = d$ and ${m_2} = {\delta^2}$, ${M_2} = {\Delta ^2}$, we have
\[
\left\{ {\sum\limits_{i = 1}^n {\eps {{({v_i})}^2}} } \right\}\left\{
{\sum\limits_{i = 1}^n {\delta {{({v_i})}^2}} } \right\} - {\left\{
{\sum\limits_{i = 1}^n {\eps ({v_i})\delta ({v_i})} } \right\}^2} \le
\frac{{{n^2}}}{4}{\left\{ {d{\Delta ^2} + r{\delta ^2}} \right\}^2}
\]
i.e., $E_1(G)\sum\limits_{i = 1}^n \delta({v_i})^2  - \xi _c(G)^2 \le \frac{n^2}{4}\left\{d\Delta ^2 + r\delta ^2\right\}^2.$

Now using (\ref{eq:6}) we have
\[
\xi _c(G)^2 \ge \frac{1}{n}E_1(G)M_1(G) - \frac{n^2}{4}\left\{
d\Delta ^2 + r\delta ^2\right\}^2
\]
which is our desired result.
\end{proof}

\begin{thm}
Let $G$ be a simple connected graph, then
\[
{\xi _c}(G) \ge \frac{2}{(n - 1)}{M_2}(G)
\]
and  equality holds if and only if $G\cong K_n$.
\end{thm}

\begin{proof}
For any $v\in V(G)$, we have $\eps(v) \ge \frac{D(v)}{(n - 1)}$, with equality if and only if $G\cong K_n$. So from the definition of modified eccentric connectivity index
\begin{equation}
\xi _c(G) = \sum\limits_{v \in V(G)}\delta(v)\eps(v)\ge
\sum\limits_{v\in V(G)}\delta(v)\frac{D(v)}{(n - 1)}.
\label{eq:8}
\end{equation}

Now since for any $v\in V(G)$, $D(v) \ge
\deg(v)$, we have from (\ref{eq:3})
\begin{equation}
{\xi _c}(G) \ge \sum\limits_{v \in V(G)} {\delta (v)\frac{\deg(v)}{(n -
1)}}  = \frac{1}{(n - 1)}\sum\limits_{v \in V(G)} {\delta (v)\deg (v)}
\label{eq:9}
\end{equation}
from where the desired result follows. Clearly, in the above relation
equality holds if and only if $G\cong K_n$.
\end{proof}

\begin{thm}
Let $G$ be a simple connected graph, then
\[
\xi_c(G) \ge 2M_1(G) - \frac{2M_2(G)}{(n - 1)}
\]
and it holds with equality if and only if $G$ is a path of length one.
\end{thm}

\begin{proof}
Since, $D(v) \ge 2n - 2 - \deg (v)$ for all $v \in V(G)$, with
equality if and only if $G \cong K_{1,n - 1}$, so from (\ref{eq:8}) we get,
\[
\xi _c(G) \ge \frac{1}{(n - 1)}\sum\limits_{v \in V(G)}\delta(v)(2n - 2 -
\deg (v))\]
\[ = \frac{1}{(n - 1)}\left[ 2(n - 1)\sum\limits_{v \in V(G)} \delta
(v) - \sum\limits_{v \in V(G)} \delta (v)\deg (v)\right]
\]
from where the desired result follows. Since the equality (\ref{eq:9}) holds if and only
if $G \cong {K_n}$, the equality holds in this result if and only if $G$ is
a path of length one which is a complete graph as well as complete bipartite
graph.
\end{proof}

\begin{thm}
Let $G$ be a simple connected graph, then
\[
\xi _c(G) \ge \delta^\frac{\delta}{\Delta}\xi ^c(G)
\]
with equality if and only if $G$ is a path of length one.
\end{thm}

\begin{proof}
Using the relationship between arithmetic and geometric mean, we
have
\[
\frac{1}{\deg (v)}\delta(v) = \frac{1}{\deg (v)}\sum\limits_{(u,v) \in
E(G)}\deg (u)  \ge \left[ \prod\limits_{(u,v) \in E(G)} \deg (u)
\right]^{\frac{1}{\deg(v)}}
\]
i.e., $\delta(v) \ge \deg(v)\delta^\frac{\delta}{\Delta}$

Thus from the definition of modified eccentric connectivity index we have,
\[
\xi _c(G) = \sum\limits_{v \in V(G)}\delta (v)\eps (v) \ge \delta^\frac{\delta}{\Delta}\sum\limits_{v \in V(G)}\deg (v)\eps (v)
\]
which is our desired result. Clearly, equality holds if and only if all the
vertices of $G$ are of same degree.
\end{proof}

Recall that the Wiener index of a connected graph $G$ (see \cite{zho10, wag10}) is
denoted by $W(G)$ and is defined as
\[
W(G) = \sum\limits_{u,v\in V(G)} d(u,v)  =
\frac{1}{2}\sum\limits_{v \in V(G)}{D(v)}
\]

\begin{thm}
Let $G$ be a simple connected graph, then
\[
\xi _c(G) \ge \frac{2\delta ^2}{n - 1}W(G)
\]
with equality if and only if $G \cong {K_n}$.
\end{thm}

\begin{proof}
Since for any $v\in V(G)$,
$\deg(v) \ge \delta $,  we have from (\ref{eq:3})
\[
\xi _c(G) \ge \frac{\delta ^2}{(n - 1)}\sum\limits_{v \in V(G)}D(v)
= \frac{2\delta^2}{(n - 1)}W(G)
\]
with equality if and only if $G \cong {K_n}$.
\end{proof}

From the definition of Harary index \cite{das09, zho08}, it follows that $W(G) \ge H(G)$,
with equality if and only if $G \cong {K_n}$. Then from the above theorem the
following Corollary follows.

\begin{cor}
Let $G$ be a simple connected graph, then
\[
\xi _c(G) \ge \frac{2\delta ^2}{n - 1}H(G)
\]
with equality if and only if $G \cong {K_n}$.
\end{cor}

We now give a Nordhaus-Gaddum type \cite{nor56} result of modified eccentric
connectivity index of connected graph.

\begin{thm}
Let $G$ be a simple connected graph with $n \ge
4$ vertices, for which the complement $\overline G$ is also connected, then
\[
\xi _c(G) + \xi _c(\overline G ) \ge 2\left[ M_1(G) + M_1(\overline G
) \right]
\]
and this holds with equality if and only if all the vertices of $G$ are of
eccentricity two.
\end{thm}

\begin{proof}
Since both $G$ and $\overline G$ are connected graph and each
has radius at least 2, from the definition of modified eccentric connectivity
index the desired result follows.
\end{proof}

\begin{thm}
Let $G$ be a $n$-vertex simple connected graph with $n
\ge 3$ vertices, then
\[
\xi _c(G) \ge (2n - 1)(n - 1)
\]
with equality if and only if $G \cong {S_n}$.
\end{thm}

\begin{proof}
A three vertex connected graph is either $S_3$ or
$K_{3}$. It may be easily checked that $\xi _c(S_3) = (2n - 1)(n -
1) \le \xi _c(K_3)$. Let $m$ be the number of edges of $G$ and
$k (0 \le k \le n)$ be the number of vertices of $G$ of degree
$n - 1$ and eccentricity one. So for these $k$ vertices, $\delta
(v) = 2m - (n - 1)$. Obviously the remaining $n - k$ vertices are of
degree less than $n - 1$ and eccentricity two.

Thus, $\xi _c(G) \ge k\left\{ 2m - (n - 1)\right\} + 2\left[ M_1(G) -
2k\left\{ 2m - (n - 1) \right\} \right]$, with equality if and only if all the
$n - k$ vertices are of degree less than $n - 1$ and of
eccentricity 2.

If $k \ge 1$, then all the vertices of $G$ except one vertex are of
degree at least $k$. So we can write, $2m \ge k(n - 1) + k(n - k)$ and
$M_1(G) \ge (n - 1)^2 + k^2(n - 1)$. Thus,
\[
\xi _c(G) \ge 2(n - 1)(n - 1 + k^2) - k\left\{k(n - 1) + k(n - k) - (n -
1)\right\},
\]
i.e., $\xi _c(G) \ge k^3 + k(n - 2) +
2(n - 1)^2.$

Clearly, the function $f(x) = x^3 + x(n - 2) + 2(n - 1)^2$ with $1 \le x \le
n$ attains the minimum value for $x = 1$, where $f(1) = (n - 1)(2n - 1)$. Hence
$\xi _c(G) \ge f(1) = (2n - 1)(n - 1)$ with equality if and only if $k =
1$ and $m = n - 1$ i.e., $G \cong {S_n}$.
\end{proof}



\bigskip
\smallskip

{\footnotesize \pn{\bf Nilanjan De}\; \\ {Department of
Basic Sciences and Humanities (Mathematics)}, {Calcutta Institute of Engineering and Management,} {Kolkata, India.}\\
{\tt Email: de.nilanjan@rediffmail.com}

\bigskip

{\footnotesize \pn{\bf Sk. Md. Abu Nayeem}\; \\ {Department of
Mathematics}, {Aliah University,} {Kolkata, India.}\\
{\tt Email: nayeem.math@aliah.ac.in}

\bigskip

{\footnotesize \pn{\bf Anita Pal}\; \\ {Department of
Mathematics}, {National Institute of Technology,} {Durgapur, India.}\\
{\tt Email: anita.buie@gmail.com}
\end{document}